\documentclass[11pt, reqno]{amsart}

\usepackage{amssymb,latexsym,amsmath,amsfonts,amsthm}
\usepackage{mathrsfs}
\usepackage{enumitem}
\usepackage[usenames]{color}
\usepackage{hyperref}
\usepackage{comment}

\allowdisplaybreaks

\numberwithin{equation}{section}

\definecolor{DPurple}{rgb}{0.46,0.2,0.69}

\theoremstyle{definition}
\newtheorem{definition}{Definition}[section]

\theoremstyle{remark}
\newtheorem{remark}[definition]{Remark}

\theoremstyle{plain}
\newtheorem{theorem}[definition]{Theorem}
\newtheorem{result}[definition]{Result}
\newtheorem{lemma}[definition]{Lemma}
\newtheorem{proposition}[definition]{Proposition}
\newtheorem{example}[definition]{Example}

\newcommand{\eps}{\varepsilon}
\newcommand{\zt}{\zeta}
\newcommand{\zbar}{\overline{z}}

\newcommand{\cHess}{\mathfrak{H}_{\raisebox{-2pt}{$\scriptstyle {\mathbb{C}}$}}}
\newcommand\Levi[1]{\mathscr{L}_{{#1}}}
\newcommand{\dbar}{\overline\partial}

\newcommand{\bdy}{\partial}

\newcommand{\D}{\mathbb{D}}
\newcommand\rB[1]{\mathbb{B}^{{#1}}}

\newcommand{\smoo}{\mathcal{C}}

\newcommand{\bcdot}{\boldsymbol{\cdot}}
\newcommand{\lrarw}{\longrightarrow}
\newcommand{\btl}{\blacktriangleleft}
\newcommand{\dst}{{\rm dist}}
\newcommand{\V}{\mathcal{V}}
\newcommand{\Tdist}{\delta_{\Omega}}
\newcommand\vc[1]{\boldsymbol{{\sf {#1}}}}

\newcommand{\Z}{\mathbb{Z}}
\newcommand{\N}{\mathbb{N}}

\newcommand{\Cn}{\mathbb{C}^n}

\newcommand{\C}{\mathbb{C}} 
\newcommand{\R}{\mathbb{R}}
\newcommand\cproj[1]{\mathbb{C}\mathbb{P}^{{#1}}}

\newcommand{\re}{{\sf Re}}
\newcommand{\im}{{\sf Im}}

\newcommand{\wt}{\widetilde}

\begin{document}

\title[Unbounded visibility domains]{Unbounded visibility domains: metric
estimates \\ and an application}

\author{Annapurna Banik}
\address{Department of Mathematics, Indian Institute of Science, Bangalore 560012,
India}
\email{annapurnab@iisc.ac.in}

\author{Gautam Bharali}
\address{Department of Mathematics, Indian Institute of Science, Bangalore 560012,
India}
\email{bharali@iisc.ac.in}

\begin{abstract}
We give an explicit lower bound, in terms of the distance from the boundary,
for the Kobayashi metric of a certain class of bounded pseudoconvex domains
in $\Cn$ with $\smoo^2$-smooth boundary using the regularity theory for the
complex Monge--Amp{\`e}re equation. Using such an estimate, among other tools,
we construct a family of unbounded Kobayashi hyperbolic domains in $\Cn$ having a certain
negative-curvature-type property with respect to the Kobayashi distance. As an
application, we prove a Picard-type extension theorem for the latter domains.
\end{abstract}


\maketitle

\vspace{-4mm}
\section{Introduction and statement of results}\label{sec:intro}
A substantial part of the effort and many of the tools discussed in this paper
are directed at the following problems that are seemingly unrelated:
\begin{enumerate}[leftmargin=25pt, label=$(\alph*)$]
  \item Using the regularity theory for the complex Monge--Amp{\`e}re equation
  on bounded domains $\Omega\Subset \Cn$, $n\geq 2$, to estimate the Kobayashi
  pseudometric $k_{\Omega}(z; \bcdot)$ in terms of ${\rm dist}(z, \bdy\Omega)$.

  \item A Picard-type extension theorem for holomorphic mappings into domains
  $\Omega\varsubsetneq \Cn$, $n\geq 2$, where $\Omega$ is unbounded, but is not
  the complement of a divisor.
\end{enumerate}
The theme that links these problems is a weak notion of negative curvature
for the metric space $(\Omega, K_{\Omega})$, where $K_{\Omega}$ denotes the
Kobayashi pseudodistance (assumed to be a distance on domains considered
in this paper). This negative-curvature-type property, called \emph{visibility},
is that, loosely speaking, geodesic lines for $K_{\Omega}$ joining two distinct
points in $\bdy{\Omega}$ must bend into $\Omega$ with some mild geometric
control (reminiscent of the Poincar{\'e} disc model of the hyperbolic plane).
\smallskip

If the metric space $(\Omega, K_{\Omega})$ is Cauchy-complete, then any two
points in $\Omega$ are joined by a geodesic (i.e., a path $\sigma : I \to
\Omega$, where $I$ is an interval, that satisfies $K_{\Omega}(\sigma(t),
\sigma(s)) = |t-s|$ for all $s,t \in I$). But when
$n \geq 2$, it is a very hard problem to tell whether, given a domain
$\Omega\varsubsetneq \Cn$, $(\Omega, K_{\Omega})$ is Cauchy-complete (even when
$\Omega$ is pseudoconvex). Therefore, for the domains considered in this paper,
$(\Omega, K_{\Omega})$ will \textbf{not} be assumed to be Cauchy-complete. Thus,
a formal definition of visibility (which will be provided in
Section~\ref{ssec:visi_Picard}) needs to be more refined than the picture described
above. This raises the question: when does a domain have the visibility
property? We begin with this discussion.

\subsection{Visibility and a Picard-type extension theorem}
\label{ssec:visi_Picard}
One of the objectives of this work is to present a new application of
visibility. This will require formalising the rough idea of visibility
mentioned above. We shall say that a domain $\Omega$ is \emph{Kobayashi hyperbolic} if
$K_{\Omega}$ is a distance.

\begin{definition}\label{D:visible} \label{defn:visi}
Let $\Omega\subset \Cn$ be a (not necessarily bounded) Kobayashi hyperbolic
domain.
\begin{enumerate}[leftmargin=25pt]
  \item Let $p$ and $q$ be two distinct points in $\bdy\Omega$. We say
  that the pair $(p,q)$ has the \emph{visibility property with respect to
  $K_\Omega$} if there exist neighbourhoods $U_p$ of $p$ and $U_q$ of $q$ in $\Cn$
  such that $\overline{U}_p\cap\overline{U}_q=\emptyset$ and such that for each
  $\lambda\geq 1$ and each $\kappa\geq 0$, there exists a compact set
  $K\subset \Omega$ such that the image of each
  $(\lambda,\kappa)$-almost-geodesic $\sigma:[0,T]\lrarw\Omega$ with
  $\sigma(0)\in U_p$ and $\sigma(T)\in U_q$ intersects $K$.
 \item We say that \emph{$\bdy\Omega$ is visible} if every pair of
 distinct points $p,q\in\bdy\Omega$ has the visibility property with respect
 to $K_\Omega$.
\end{enumerate}
\end{definition}

The property of $\bdy\Omega$ being visible is closely related to the notion of
$\Omega$ being a \emph{visibility domain}, which was introduced by
Bharali--Zimmer \cite{bharali-zimmer:2017, bharali-zimmer:2023}. The
two notions are equivalent. This, in brief, is due to the fact that $\overline{\Omega}$
(resp., the Freudenthal end-compactification of $\overline{\Omega}$) is sequentially
compact under the assumptions made in \cite{bharali-zimmer:2017} (resp., in
\cite{bharali-zimmer:2023}). For an alternative argument, see
\cite[Section~1.3]{masanta:vdekdcm24}. We will not define visibility domains
here (as they are not germane to the discussion). Instead, we will work with
the properties introduced in Definition~\ref{defn:visi}, which are adequate for the
application presented here. The notion of visibility itself is not
new: the property introduced in Definition~\ref{defn:visi} is reminiscent of
a property introduced by Eberlein--O'Neill in \cite{eberlein-oneill:1973} in the
context of Riemannian manifolds having non-positive sectional curvature\,---\,formulated
in terms of an abstract boundary in place of $\bdy\Omega$ and geodesics in place of 
$(\lambda,\kappa)$-almost-geodesics. This property is also seen in proper geodesic
metric spaces that are Gromov hyperbolic; 
in this setting, the Gromov boundary takes the place of $\bdy\Omega$.
The latter form of visibility, for domains $\Omega\varsubsetneq \Cn$
such that $(\Omega, K_{\Omega})$ is a proper (hence geodesic, by the properties of
$K_{\Omega}$) metric space, underlies the proofs of several results that are precursors
to the results on holomorphic mappings alluded to in the next paragraph: see,
for instance, \cite{balogh-bonk:2000, bracci-gaussier:2020} and
\cite[Part~II]{karlsson:2005}. Results of the latter description are also given
by \cite{zimmer:2017}, which are more directly linked to
the ideas in \cite{eberlein-oneill:1973}. Also see \cite{khanh-thu:2016} for
a result on iterative dynamics whose proof relies on a property that could
be deduced from the visibility of $\bdy\Omega$ (but instead relies on
\cite{karlsson:2005}). In the results just cited,
$(\Omega, K_{\Omega})$ is assumed to be a geodesic space. But recall the discussion on
the difficulty in knowing when $(\Omega, K_{\Omega})$ admits geodesics. This
explains the role of $(\lambda, \kappa)$-almost-geodesics (see Section~\ref{sec:hyp_imb}
for a definition) in Definition~\ref{defn:visi}. They serve as substitutes for
geodesics: this is because if $\Omega$ is Kobayashi hyperbolic, then (regardless
of whether $(\Omega, K_{\Omega})$ is Cauchy-complete) for any
$z,w \in \Omega$, $z \neq w$, and
any $\kappa>0$, there exists a $(\lambda, \kappa)$-almost-geodesic joining $z$ and
$w$ \cite[Proposition~5.3]{bharali-zimmer:2023}. 
\smallskip

Visibility of $\bdy\Omega$ has been used to deduce properties of holomorphic
mappings into $\Omega$\,---\,ranging from their continuous extendability,
to the iterative dynamics of such self-maps\,---\,which are
too numerous to mention here. Instead, we refer readers to
\cite{bharali-zimmer:2017,
bracci-nikolov-thomas:2022, chandel-maitra-sarkar:2021, bharali-zimmer:2023}.
Given this, it is desirable to identify families of unbounded domains
$\Omega\varsubsetneq \Cn$ such that $\bdy\Omega$ is visible. A rich
collection of \textbf{planar} domains with the latter property that also
satisfy other metrical conditions, and domains in $\Cn$, $n\geq 2$, with the
latter property and having rather wild boundaries, have been constructed
in \cite{bharali-zimmer:2023}. But, \emph{given an unbounded domain
$\Omega\varsubsetneq \Cn$, $n\geq 2$, such that $\bdy\Omega$ is
$\smoo^2$-smooth, Levi pseudoconvex, but \textbf{not} strongly Levi
pseudoconvex, are there conditions under which $\bdy\Omega$ is visible?} One
of our theorems addresses this natural question. Why the interest in
\textbf{unbounded} domains, one may ask. The answer will be evident when we
discuss Picard-type theorems.
\smallskip

We present a condition for the visibility of $\bdy \Omega$ that answers the
question in italics stated above. Our condition, roughly, is that the set of
points at which $\bdy \Omega$ is weakly Levi pseudoconvex, if non-empty, is
small but \textbf{not} necessarily totally disconnected.
Some notation: if $A$
and $B$ are non-negative quantities, $A \gtrsim B$ will mean that there
exists a constant $c>0$, independent of all variables determining $A$ and
$B$, such that $A \geq c B$. The vector bundle 
$H(\bdy\Omega) := T(\bdy\Omega)\cap i \,T(\bdy\Omega)$; so, 
$H_{\xi}(\bdy\Omega)$ is the maximal complex subspace of $T_{\xi}(\bdy \Omega)$:
the tangent space of $\bdy\Omega$ at $\xi$. We now define the \emph{Levi form}
of $\bdy\Omega$, denoted by $\Levi{\Omega}$. While, abstractly, the Levi form is
a vector-valued quadratic form $\Levi{\Omega}(\xi, \bcdot) : H_{\xi}(\bdy\Omega)
\to T_{\xi}(\bdy\Omega)\otimes\C/H_{\xi}(\bdy\Omega)\otimes\C$ for $\xi\in
\bdy\Omega$\,---\,see, for
instance, \cite[Chapter~10]{boggess:1991}\,---\,since $\bdy\Omega$ is a
CR~hypersurface embedded in $\Cn$, we can define $\Levi{\Omega}(\xi, \bcdot)$ to
be $\R$-valued. This definition makes use of the standard (flat) Hermitian
metric on $T(\Cn)\otimes\C$, restricted to $T(\bdy\Omega)\otimes\C$, to
identify $T_{\xi}(\bdy\Omega)\otimes\C/H_{\xi}(\bdy\Omega)\otimes\C$ with\linebreak
$(T_{\xi}(\bdy\Omega)\otimes\C)\ominus (H_{\xi}(\bdy\Omega)\otimes\C)$, the
orthogonal complement being given by the above-mentioned metric. Let
$\eta_{\xi}$ be the outward unit normal vector to $\bdy\Omega$ at 
$\xi$ and let $\mathbb{J}_z$\,$(=\mathbb{J}$ for each
$z\in \Cn$) denote the standard almost complex structure on
$T_z(\Cn)\otimes\C$ for each $z\in \Cn$. Since $\mathbb{J}(\eta_{\xi})$
spans $(T_{\xi}(\bdy\Omega)\otimes\C)\ominus (H_{\xi}(\bdy\Omega)\otimes\C)$, the last
observation enables us to define the Levi form as   
\[
  \Levi{\Omega}(\xi; v) :=
  (1/2i)\,\big\langle\,[\,\overline{\vc{v}}, \vc{v}\,]_{\xi},\,\mathbb{J}(\eta_{\xi})
  \big\rangle_{\xi} \quad 
  \forall v\in H_{\xi}(\bdy\Omega) 
  \text{ and } \forall \xi \in \bdy{\Omega},
\]
where $\langle\bcdot\,, \bcdot\rangle_{\xi}$ is the above-mentioned flat
metric on $T_{\xi}(\bdy\Omega)\otimes\C$ and, if $v = (v_1,\dots, v_n)$,
$\vc{v}$ is any $\smoo^1$-smooth section of $H^{1,0}(\bdy\Omega)$ defined
around $\xi$ such that
$\vc{v}(\xi) = \sum_{1\leq j\leq n}v_j\big(\left.
\partial/\partial z_j\right|_{\xi}\big)$. It is easy to
see that the right-hand side above does not depend on the choice of $\vc{v}$.
We must mention that the choice of the frame field $\xi\longmapsto \eta_{\xi}$
of $(T(\bdy\Omega)\otimes\C)\ominus (H(\bdy\Omega)\otimes\C)$ is such that if
$\Omega$ is Levi pseudoconvex, then $\Levi{\Omega}\geq 0$. Then, $w(\bdy\Omega)$
is the set of points in $\bdy\Omega$ at which $\bdy\Omega$ is weakly Levi
pseudoconvex: i.e., the set of all $\xi\in \bdy\Omega$ at which
$\Levi{\Omega}(\xi,\bcdot)$ is \textbf{not} strictly positive definite.%

\begin{theorem} \label{th:unbdd_visibile}
Let $\Omega \subset \Cn$, $n\geq 2$, be an unbounded Kobayashi hyperbolic
domain that is pseudoconvex and has $\smoo^2$-smooth boundary. Suppose
there exists a $\smoo^2$-smooth closed $1$-submanifold $S$ of $\bdy\Omega$
such that $w(\bdy\Omega)\subset S$.
Assume that for each $p\in w(\bdy\Omega)$, there exists a neighbourhood
$U_p$ of $p$ and $m_p > 2$ such that
\begin{equation}\label{eqn:Levi}
  \Levi{\Omega}(\xi;v) \gtrsim \dst(\xi, S)^{m_p-2}\|v\|^2 
  \quad \forall v\in H_{\xi}(\bdy\Omega) 
  \text{ and } \forall \xi \in (\bdy{\Omega}\cap U_p)\setminus S.
\end{equation}
Then, $\bdy\Omega$ is visible.    
\end{theorem}

Before we can present our next result, we need a general definition.

\begin{definition} \label{defn:hyp_imb}
Let $Z$ be a complex manifold and $Y$ a complex submanifold of $Z$. We say that
$Y$ is \emph{hyperbolically imbedded} in $Z$ if for every point $p \in
\overline{Y}$ and
for each neighbourhood $U_p$ of $p$ in $Z$, there exists a neighbourhood $V_p$
of $p$ in $Z$ with $V_p \Subset U_p$ such that 
$K_Y \big( \overline{V_p} \cap Y, Y \setminus U_p \big)  >0$. Here,
$\overline{V_p}$ is the closure of $V_p$ in $Z$.
\end{definition}

The property of being hyperbolically imbedded is relevant to a class of extension
theorems that we wish to examine further. The archetypal results of this class are:

\begin{result}[Kiernan, \cite{kiernan:1972}: paraphrased for $Y$, $Z$
manifolds] \label{r:Picard-type}
Let $Z$ be a complex manifold and let $Y \subset Z$ be a hyperbolically
imbedded relatively compact submanifold.
\begin{enumerate}[leftmargin=25pt]
  \item Then, every holomorphic map $f:\D^* \to Y$
  extends as a holomorphic map $\widetilde{f}:\D \to Z$.
\smallskip

  \item Let $X$ be a complex manifold, let
  $k = \dim_{\C}(X)$, and let $\mathcal{A} \varsubsetneq X$ be an analytic
  subvariety of $X$ of dimension $(k-1)$ having at most normal-crossing
  singularities. Then, any holomorphic map $f: X \setminus \mathcal{A}
  \rightarrow Y$ extends as a holomorphic map $\widetilde{f}:X \rightarrow Z$.
\end{enumerate}
\end{result}

Kwack had established a result of a similar nature under an analytical
hypothesis \cite[Theorem~3]{kwack:1969}, whose proof is repurposed in
\cite{kiernan:1972} to prove $(1)$ above. 
Theorems such as Result~\ref{r:Picard-type} are called
\emph{Picard-type extension theorems}. The reason for this terminology is
as follows: if $Z= \C \mathbb{P}^1$ and $Y= \C\setminus \{0, 1\}$, then $(1)$
is implied by the Big Picard Theorem.
\smallskip

It is well known that the complement of $(2n + 1)$ hyperplanes in general
position in $\C\mathbb{P}^n$ is hyperbolically imbedded in $\C\mathbb{P}^n$.
To the best of our knowledge, no general techniques are known that
tell us when $Y$ is hyperbolically imbedded in $Z$ with $\dim_{\C}(Z)$ being
arbitrary\,---\,for $Y,\, Z$ as in Definition~\ref{defn:hyp_imb}\,---\,beyond
cases where $Z = \C\mathbb{P}^n$ and
$Y$ are complements of certain divisors in $\C\mathbb{P}^n$ (in which case one relies on
\cite{kiernan:1973} by Kiernan).
It is thus natural to ask: what are some other \emph{explicit} geometric conditions on the
pair $(Y,Z)$ that would yield the same conclusions as Result~\ref{r:Picard-type}.
A good place to start would be to take $Z= \Cn$ and $Y$ a domain in $\Cn$. But observe
that if $Y \varsubsetneq \C^n$ is bounded, the extension problem becomes trivial due
to Riemann's removable singularities theorem. This is why we consider unbounded domains
in our next theorem.    

\begin{theorem} \label{th:unbdd_Picard-type}
Let $\Omega \subset \Cn, n \geq 2$, be an unbounded Kobayashi hyperbolic domain
with the properties stated in Theorem~\ref{th:unbdd_visibile}. Let $X$ be a
complex manifold, let $k = \dim_{\C}(X)$, and let $\mathcal{A} \varsubsetneq X$
be an analytic subvariety of dimension $(k-1)$ having at most normal-crossing
singularities. Then, any holomorphic map $f: X \setminus \mathcal{A}
\rightarrow \Omega$ extends as a continuous
map $\widetilde{f}:X \rightarrow \overline{\Omega}^{\infty}$.
\end{theorem}

Here, $\overline{\Omega}^{\infty}$ denotes the closure of
$\Omega$ relative to the one-point compactification of $\Cn$.
Both Theorems~\ref{th:unbdd_visibile} and~\ref{th:unbdd_Picard-type}
feature the same domain $\Omega$. This is because, as in
Result~\ref{r:Picard-type}, $\Omega$ being hyperbolically imbedded continues to
be crucial to $f: X \setminus \mathcal{A}\rightarrow \Omega$ admitting an
extension with any degree of regularity. The latter condition follows
if $\bdy\Omega$ is visible; see Proposition~\ref{prpn:visi_hyp-im}.
This\,---\,in
view of the discussion preceding Theorem~\ref{th:unbdd_Picard-type}\,---\,is the
reason for our interest in identifying unbounded visibility domains.
Since $\Omega$ in Theorem~\ref{th:unbdd_Picard-type} is not relatively compact,
it does not follow from Result~\ref{r:Picard-type}. Instead, we rely on the work
of Joseph--Kwack \cite{joseph-kwack:1994}; see
Result~\ref{r:Picard-type_hyp-imb}. Their work does not, however, provide any
geometric conditions for a domain $\Omega$, whether in $\Cn$ or in some complex
manifold, to be hyperbolically imbedded. The focus of
\cite{joseph-kwack:1994} is a set of function-theoretic
characterisations of hyperbolic imbedding. It is, therefore, natural to seek
\textbf{geometric} conditions for a domain $\Omega\varsubsetneq \Cn$ to be
hyperbolically imbedded. The hypothesis of
Theorems~\ref{th:unbdd_visibile} and~\ref{th:unbdd_Picard-type} provides
just such a condition.
Domains satisfying this hypothesis are abundant: see, for
instance, the examples given by Gaussier in 
\cite[Section~3.2]{gaussier:1999}. We conclude this section with one last
question: could one extend Theorem~\ref{th:unbdd_Picard-type} to domains
$X \setminus \mathcal{A}$ such that $\mathcal{A}$ has worse singularities?
We can show that\,---\,in the notation of
Theorem~\ref{th:unbdd_Picard-type}\,---\,a holomorphic map
$f: X \setminus\mathcal{A}\rightarrow \Omega$ with $\Omega$ hyperbolically
imbedded does not, in general, extend continuously to $X$ if the singularities
of $\mathcal{A}$ are slightly worse than normal-crossing singularities; see
Example~\ref{ex:not_ext}.%
\smallskip

\subsection{Lower bounds for the Kobayashi pseudometric}
\label{ssec:Kobayashi}
Estimates for the Kobayashi pseudometric are an essential tool for the project
discussed above. This motivates our next theorem and Theorem~\ref{th:kob_lower_bnd_general}. These
results are inspired by a well-known
result of Diederich--Forn{\ae}ss \cite[Theorem~4]{diederich-fornaess:1979}, which
provides a lower bound for the Kobayashi pseudometric $k_{\Omega}$ of a bounded
pseudoconvex domain $\Omega\subset \Cn$ with real-analytic boundary. The lower
bound that they deduce for $k_{\Omega}(z; v)$,
$(z,v)\in \Omega\times \Cn$, is in terms of some positive power of
$(1/\dst(z, \bdy\Omega))$, and has many applications. These applications are,
in part, the reason for our interest in such an estimate on domains with just
$\smoo^2$-smooth boundary. The notation in the theorem below is as described prior
to Theorem~\ref{th:unbdd_visibile}. For any $z\in \Omega$, we shall abbreviate
$\dst(z, \bdy\Omega)$ to $\Tdist(z)$.
\smallskip

\begin{theorem} \label{th:kob_lower_bnd}
Let $\Omega \subset \Cn$, $n\geq 2$, be a bounded pseudoconvex domain having
$\smoo^2$-smooth boundary. Assume that there exists a $\smoo^2$-smooth closed
submanifold of $\bdy\Omega$ such that $S$ is totally-real and such that
$w(\bdy\Omega)\subset S$. Suppose there exists a number $m > 2$ such that
\[
  \Levi{\Omega}(\xi;v) \gtrsim \dst(\xi, S)^{m-2}\|v\|^2 
  \quad \forall v\in H_{\xi}(\bdy\Omega) 
  \text{ and } \forall \xi \in \bdy{\Omega}\setminus S.
\]
Then,
there exists a constant $c > 0$ such that
\begin{equation} \label{eqn:k-metric_lb}
  k_{\Omega}(z; v) \geq 
  c\,\frac{\|v\|}
            {\Tdist(z)^{1/m}}
            \quad \forall z\in \Omega \text{ and } \forall v \in \Cn.
\end{equation}
\end{theorem}

On a first reading, the estimate \eqref{eqn:k-metric_lb} might seem
unsurprising. However, to the best of our knowledge, estimates of the form
\eqref{eqn:k-metric_lb} seen in the literature that are well-argued do not,
for $\Omega$ weakly pseudoconvex, provide an \textbf{explicit} exponent
of $\delta_{\Omega}$. Moreover, there are other significant reasons for placing
the estimate \eqref{eqn:k-metric_lb} on record. Namely:
\begin{itemize}
    \item For bounded, weakly pseudoconvex, finite-type domains $\Omega$ with
    $\bdy\Omega$ \textbf{not} real analytic, lower bounds for $k_{\Omega}$
    resembling \eqref{eqn:k-metric_lb} have been (re)asserted on many
    occasions\,---\,the earliest instance being
    \cite{cho:1992}. Each such claim has, eventually, relied on
    the difficult half of the paper \cite{catlin:1984}
    by Catlin. There seems to be a certain deficit in understanding the latter
    work\,---\,nor is there any alternative exposition on the efficacy of a
    construction, called a \emph{boundary system}, on which the proofs of
    the above-mentioned assertions rely.%
    \smallskip
    
    \item We introduce a method relying on the regularity theory for the
    complex Monge--Amp{\`e}re equation to derive lower bounds of the form
    \eqref{eqn:k-metric_lb}. One way of deriving such a bound 
    is to construct plurisubharmonic peak functions satisfying certain precise
    estimates: see, for instance, \cite[Theorem~2]{diederich-fornaess:1979}, 
    \cite[Proposition~4.2]{catlin:1989}. Similar peak functions, but with
    \textbf{less restrictive} requirements, suffice to prove the existence
    and regularity of solutions of the complex Monge--Amp{\`e}re equation, as
    one would expect from the proof of \cite[Theorem~6.2]{bedford-taylor:1976}
    (see Remark~\ref{rem:MA-eqn_why}).
    This underlies\,---\,in view of a result of
    Sibony \cite{sibony:1981}\,---\,an effective and simpler idea
    for deriving lower bounds for the Kobayashi metric.
\end{itemize}

The latter point is substantiated by a general result whose proof (similar to
that of Theorem~\ref{th:kob_lower_bnd}) is the method alluded to. For its exact
statement, we refer the reader to Theorem~\ref{th:kob_lower_bnd_general}.%
\bigskip

\section{Preliminaries on hyperbolic imbedding and visibility}
\label{sec:hyp_imb}
This section is devoted to assorted observations of a technical nature that
will be needed in our discussion surrounding
Theorems~\ref{th:unbdd_visibile} and~\ref{th:unbdd_Picard-type}. But we first
explain some notation used below and in later sections (some of which has also
been used without comment in Section~\ref{sec:intro}).
\smallskip

\subsection{Common notations}
\begin{enumerate}[leftmargin=25pt]
  \item For $v \in {\R}^d $, $\|v\|$ denotes the
  Euclidean norm. For any $x \in {\R}^d$ and $A \subset {\R}^d$, we write
  $\text{dist}(x, A)\,:=\,
  \inf \{\|x-a\|: a \in A\}$. 
  \smallskip
  \item Given a point $x \in \R^d$ and $r>0$, ${\mathbb{B}}^d(x,r)$ denotes the open
  Euclidean ball in $\R^d$ with radius $r$ and center $x$.
  \smallskip
  \item Given a point $z \in \Cn$ and $r>0$, $B^n(z,r)$ denotes the open
  Euclidean
  ball in $\Cn$ with radius $r$ and center $z$. For simplicity, we write
  $\D \,:=\, B^1(0,1)$. Also, we write $\D^*:= \D \setminus \{0\}$.
  \smallskip
  \item Given a $\smoo^2$-smooth function $\phi:\Omega \to \C$ defined in some domain
  $\Omega \subset \Cn$, $ (\cHess{\phi})(z)$ denotes the complex Hessian of $\phi$
  at $z \in \Omega$.
  \item $\langle\bcdot\,, \bcdot\rangle$ denotes the standard Hermitian inner
  product on $\Cn$.
\end{enumerate}

\subsection{Definitions and results}
We begin with an elementary fact:

\begin{lemma} \label{l:hyp_imb_intermediate}
Let $Z$ be a complex manifold. Let $X$ and $Y$ be domains in $Z$ such that
$X \varsubsetneq Y \varsubsetneq Z$. If $X$ is hyperbolically imbedded in
$Z$, then $X$ is hyperbolically imbedded in $Y$ as well.
\end{lemma}

Let $p$, $U_p$ and $V_p$ be as in Definition~\ref{defn:hyp_imb}; by the fact
that the closure of $V_p\cap Y$ in $Y$ equals
$(\overline{V_p \cap Y}) \cap Y$ (where $\overline{V_p \cap Y}$ denotes the
closure in $Z$), the above result follows immediately.
\smallskip

The remainder of this section focuses on definitions and facts related to the
property of visibility of $\bdy\Omega$, $\Omega\subset \Cn$ being a domain.

\begin{definition} \label{defn:almost-geodesic}
Let $\Omega \subset \C^n$ be a domain and let $I \subset \R$ be an interval.
For $\lambda \geq 1$ and $\kappa \geq 0$, a curve $\sigma: I \to \Omega$ is
called a \emph{$(\lambda, \kappa)$-almost-geodesic} if
\begin{itemize}
  \item ${\lambda}^{-1}|t-s| - \kappa \leq K_{\Omega}(\sigma(s), \sigma(t))
  \leq \lambda |t-s| + \kappa$ for every $s,t \in I$, and
  \smallskip
  \item $\sigma$ is absolutely continuous (whereby $\sigma'(t)$ exists for
  almost every $t \in I$) and $k_{\Omega}(\sigma(t); \sigma'(t))
  \leq \lambda$ for almost every $t \in I$.
\end{itemize}
\end{definition}

Next, we present a definition that formalises one of the sufficient conditions
on a domain $\Omega\subset \Cn$ under which $\bdy\Omega$ is visible. It is an
adaptation, introduced by Bharali--Zimmer \cite{bharali-zimmer:2023}, to
unbounded domains of a well-known property.

\begin{definition} \label{defn:loc-int-cone}
Let $\Omega \subset \Cn$ be a domain. We say that $\Omega$ satisfies a
\emph{local interior-cone condition} if for each $R>0$ there exist constants
$r_0>0$, $\theta \in (0, \pi)$, and a compact subset $K \subset \Omega$,
which depend on $R$, such that for each $z \in B^n(0,R) \cap 
(\Omega \setminus K)$, there is a point $\xi_z \in \bdy \Omega$ and a unit
vector $v_z$ such that
\begin{itemize}
  \item $z= \xi_z + t v_z$ for some $t \in (0,r_0)$, and
  \smallskip
  
  \item $(\xi_z + \Gamma(v_z, \theta)) \cap B^n(\xi_z, r_0) \subset \Omega$.
\end{itemize}
Here, $\Gamma(v_z, \theta)$ denotes the open cone 
$$
  \Gamma(v_z,\theta) \,:=\, \{ w \in {\C}^n : \re \, \langle w,v_z \rangle > \cos (\theta /2)\,  \|w\| \}.
$$
\end{definition}

The following result is classical in the case when $\Omega$ is bounded. But the
choice of $K = K(R)$ in Definition~\ref{defn:loc-int-cone} requires care when
$\Omega$ is unbounded, for which reason we provide a proof.

\begin{lemma} \label{l:loc-int-cone}
Let $\Omega \subset \C^n$ be an unbounded domain with $\smoo^2$-smooth
boundary. Then, $\Omega$ satisfies a local interior-cone condition.
\end{lemma}

\begin{proof}
Since $\bdy \Omega$ is $\smoo^2$-smooth, for each $p \in \bdy \Omega$, we can find
balls $W_p:= B^n(p, R_p), \,
V_p:=  B^n(p, r_p)$ with $0 < r_p < R_p$ such that the following holds:
\begin{enumerate}[leftmargin=25pt, label=$(\alph*)$]
  \item For each $z \in V_p \cap \Omega$, there exists unique
  $\xi_z \in \bdy \Omega$ such that $\| \xi_z - z \| = \delta_{\Omega}(z)$
  and $\xi_z \in W_p$,

  \smallskip

  \item If $\xi_1 \neq \xi_2 \in \bdy \Omega \cap  B^n(p, R_p)$, then $\{\xi_1 + t \eta_{\xi_1}: t \geq 0 \}
  \cap \{\xi_2 + t \eta_{\xi_2}: t \geq 0 \} \cap B^n(p, r_p) = \emptyset$ (where $\eta_{\xi}$ denotes
  the inward unit normal at $\xi \in \bdy \Omega$).
\end{enumerate}
\smallskip

Fix $R>0$. If $\overline{B^n(0,R)}\cap \Omega = \emptyset$ or if $\overline{B^n(0,R)}
\subset \Omega$, then the two conditions in Definition~\ref{defn:loc-int-cone}
hold true vacuously (taking $K = \overline{B^n(0,R)}$ in the latter case).
Hence, fix $R>0$ such that $\bdy \Omega \cap \overline{B^n(0, R)} 
\neq \emptyset$. Write $ S:= \bdy \Omega \cap \overline{B^n(0, R)}$. Let
$W := \bigcup_{p \in S} W_p$ and
$V:=\bigcup_{p\in S} V_p $. As $S$ is compact, we can find a finite subcover $\{V_1, \cdots, V_k\}$ of
$\{V_p: p \in S\}$ that covers $S$. Write $ V_j := B^n(p_j, r_j)$ and $W_j := B^n(p_j, R_j)$. We can
choose $r>0$ sufficiently small such that
\begin{align} \label{eqn:choice-of-r}
  \overline{\bigcup\nolimits_{p \in S} B^n(p,r)} \subset \bigcup\nolimits_{j=1}^{k} V_j.
\end{align}

Let $K$ be the compact set, $K \subset \Omega$, defined as follows:
$$
  K:=\overline{B^n(0,R)} \bigcap \big( \overline{\Omega} \setminus \bigcup\nolimits_{p \in S} B^n(p, r/2) \big).
$$
Let $r_0:= r/2$. (Note that $K$ and $r_0$ depend on $R$.)
\smallskip

\textbf{Fix} $z \in B^n(0,R) \cap (\Omega \setminus K)$. Then, $z \in \bigcup_{p \in S} B^n(p, r/2)$. Hence,
by \eqref{eqn:choice-of-r}, $z \in \bigcup_{j=1}^{k} V_j$. Thus, there exists a unique point in $\bdy \Omega$,
call it $\xi_z$, such that $\delta_{\Omega}(z)= \| z - \xi_z\|$. Let $\eta_{\xi_z}$ denote the inward unit
normal vector to $\bdy \Omega$ at $\xi_z$. Let $z':= \xi_z + (r/2) \eta_{\xi_z}$. If $p$ is a point in $S$
such that $z\in B^n(p, r/2)$, then it is immediate that:
\begin{itemize}
  \item $\|z- z'\| = |r/2 - \Tdist(z)|$, whereby $z'\in B^n(p, r)$.
  \item If, for some $j=1,\dots, k$, $V_j$ contains $z'$ (owing to \eqref{eqn:choice-of-r}), then
  (with $\xi_{z'}$ having a meaning analogous to $\xi_{z}$) $\xi_{z}, \xi_{z'}\in B^n(p_j, R_j) =: W_j$.
\end{itemize}
Thus, by the property of the pair $(V_j, W_j)$ given by $(b)$ above, $\xi_{z} = \xi_{z'}$; thus
$\delta_{\Omega}(z')= \|z'-\xi_z\|= r/2$. Therefore, $B^n(z', r/2) \subset \Omega$. 
\smallskip

Now, clearly, $z= \xi_z + t \eta_{\xi_z}$, where $t= \|z - \xi_z\|< r/2=r_0$. Also, it is easy to see that
there exists a uniform $\theta\in (0, \pi)$ such that
$$
  (\xi_z + \Gamma(\eta_{\xi_z}, \theta)) \cap B^n(\xi_z, r_0) \subset B^n(z', r_0) \cap B^n(\xi_z, r_0)
  \subset \Omega.
$$
Here, $\theta$ is given by the following:
\begin{align*}
  \cos(\theta/2) &= {\re \langle \eta_{\xi_z}, v\rangle}/\|v\| \\
  &= {\re \langle \eta_{\xi_z}, v\rangle}/r_0, 
  \quad \text{where } v\in \bdy{B^n(\xi_z, r_0)}\cap \bdy{B^n(z', r_0)}.
\end{align*}
In the above expression, $\theta$ is independent of the choice of $v$ as the inner product depends only on
$r_0$. For this reason, $\theta$ is also independent of $z\in B^n(0,R) \cap (\Omega \setminus K)$. This
establishes the conditions given in Definition~\ref{defn:loc-int-cone}.
\end{proof}

The next result is a version of a result due to Sarkar 
\cite[Proposition~3.2-(3)]{sarkar:2023}. However, unlike in
\cite[Proposition~3.2-(3)]{sarkar:2023}, we are given that $\bdy\Omega$ is visible as a part of the
hypothesis of the result below. This results
in a simpler proof than in \cite{sarkar:2023}. Since it is so vital to
proving Theorem~\ref{th:unbdd_Picard-type}, we shall provide a proof of the
following

\begin{proposition} \label{prpn:visi_hyp-im}
Let $\Omega \subset \C^n$ be an unbounded Kobayashi hyperbolic domain and suppose $\bdy\Omega$ is
visible. Then, $\Omega$ is hyperbolically imbedded in $\Cn$.
\end{proposition}

\begin{proof}
Let $p \in \bdy \Omega$. Fix a pair of bounded $\Cn$-neighbourhoods $U_p, V_p$ of $p$ such that $V_p
\Subset U_p$. It suffices to show that
$K_\Omega \big(\overline{V_p} \cap \Omega, \Omega \setminus U_p \big)>0.$
\smallskip

We prove the above by contradiction. Assume that $K_\Omega \big(\overline{V_p} \cap \Omega, \Omega \setminus
U_p \big)=0$. Then, there exist a pair of sequences $\{z_{\nu} \} \subset \overline{V_p} \cap \Omega$ and
$\{w_{\nu}\} \subset \Omega \setminus U_p $ such that
$K_{\Omega}(z_{\nu}, w_{\nu}) \to 0$ as $\nu \to \infty$.
As $\Omega$ is Kobayashi hyperbolic, by \cite[Proposition~5.3]{bharali-zimmer:2023}, for each $\nu$ there
exists a $(1,1/{\nu})$-almost-geodesic $\sigma_{\nu}:[a_{\nu}, b_{\nu}] \to \Omega$ joining $z_{\nu}$ and
$w_{\nu}$.

\medskip
\noindent{\textbf{Claim.}} There exist a subsequence $\{(z_{\nu_k},
w_{\nu_k})\}$ of $\{(z_{\nu}, w_{\nu})\}$ and a compact $K \subset \Omega$
such that $\sigma_{\nu_k} ([a_{\nu_k}, b_{\nu_k}])  \cap K \neq
\emptyset$ for all $k$.
\smallskip

\noindent{\emph{Proof of claim:}} Suppose $\{z_{\nu}: \nu \in \Z_+\} \Subset
\Omega$. Let $K:= \overline{\{z_{\nu}: \nu \in \Z_+\}}$, which is
contained in $\Omega$ and is compact, since $V_p$ is bounded. Clearly,
$\sigma_{\nu}([a_{\nu}, b_{\nu}])\cap K \neq \emptyset$ for all $\nu$.
\smallskip

Now, suppose $\overline{\{z_{\nu}: \nu \in \Z_+\}} \nsubseteq \Omega$. Then,
passing to a subsequence and relabelling, if needed, we may assume that
$z_{\nu} \to \xi$, for some $\xi \in \bdy \Omega \cap \overline{V_p}$.
For each $\nu$, define
$$
  t_{\nu}:= \inf \{t \in [a_{\nu}, b_{\nu}]: \sigma_{\nu}(t) \in \Omega
  \setminus U_p \}.
$$
Clearly, $ t_{\nu} \in (a_{\nu}, b_{\nu})$ and $\sigma_{\nu}(t_{\nu}) \in 
\bdy U_p \cap \Omega$. Write $\zeta_{\nu}:= \sigma_{\nu}(t_{\nu})$. As before,
if $\{\zeta_{\nu}: \nu \in \Z_+\} \Subset \Omega $, then
$K:= \overline{\{\zeta_{\nu}: \nu \in \Z_+\}}$ is our desired compact set that
intersects the image of $\sigma_{\nu}$ for each $\nu$. If not, then we get a
subsequence $\{\zeta_{\nu_k}\}$ of $\{\zeta_{\nu}\}$ and a point $\eta \in
\bdy \Omega \cap \bdy U_p$ such that $\zeta_{\nu_k} \to \eta$. Clearly,
$\eta \neq \xi = \lim_{k\to \infty}z_{\nu_k}$. Observe that 
$\wt{\sigma}_k := \sigma_{\nu_k}|_{[a_{\nu_k}, t_{\nu_k}]}$ is a
$(1,1)$-almost-geodesic joining $z_{\nu_k}$ and $\zeta_{\nu_k}$. Thus,
as $\bdy \Omega$ is visible and as $\xi$ and $\eta$ are distinct
boundary points, there is a compact $K \subset \Omega$ such that
${\sf image}(\wt{\sigma}_k)\cap K \neq \emptyset$ for every $k$
sufficiently large, from which the claim follows. \hfill $\btl$
\smallskip

Now, let $o_{k}:=\wt{\sigma}_{k}(s_{k}) \in {\sf image}(\wt{\sigma}_{k})\cap K$.
Without loss of generality, we can assume that there is a point $o \in K$ such that
$o_{k} \to o$. Using the fact that $\sigma_{{\nu}_k}:
[a_{\nu_k}, b_{\nu_k}] \to \Omega$ is a $(1,1/\nu_k)$-almost-geodesic, we get
\begin{align*}
  K_{\Omega}(z_{\nu_k}, o_{k}) + K_{\Omega}( o_{k}, w_{\nu_k}) 
  &\leq (s_{k} - a_{\nu_k}) + (b_{\nu_k} - s_{k}) + 2/\nu_k \\
  &=(b_{\nu_k} - a_{\nu_k}) + 2/{\nu_k} \leq 
  K_{\Omega}(z_{\nu_k}, w_{\nu_k}) + 3/\nu_k
  \quad \forall k.
\end{align*}
By assumption, the right-hand side of the above inequality goes to $0$ as
$k \to \infty$. Then, as $\Omega$ is Kobayashi hyperbolic, we must have
$$
  \lim_{k\to \infty} z_{\nu_k} = \lim_{k\to \infty} w_{\nu_k}= \lim_{k\to \infty} o_{k} = o,
$$
which is impossible since $(\overline{V}_p\cap \overline{\Omega})\cap
(\overline{\Omega}\setminus U_p) = \emptyset$.  We have arrived at a
contradiction. This proves the result.
\end{proof}

\section{Analytical preliminaries} \label{sec:analytic}
This section is devoted to definitions and results that will be essential to
the proofs of Theorems~\ref{th:kob_lower_bnd} and~\ref{th:unbdd_visibile}.
Recall that the exterior derivative $d = (\bdy + \dbar)$ and
$d^c := i(\bdy - \dbar)$. Let $\Omega\subset \Cn$, $n\geq 2$, be a bounded
domain. Given two functions $\phi \in \smoo(\bdy\Omega; \R)$ and 
$h\in \smoo(\Omega; \R)$, $h\geq 0$, the \emph{Dirichlet problem for the complex
Monge--Amp{\`e}re equation} is the non-linear boundary-value problem that seeks
a function $u \in \smoo(\overline\Omega; \R)$ such that $u|_{\Omega}$ is
plurisubharmonic (which we shall denote as 
$u\in {\rm psh}(\Omega)\cap \smoo(\overline\Omega)$ such that
\begin{align}
  (dd^c{u})^n\,
  :=\,\underbrace{dd^c{u}\wedge\dots \wedge dd^c{u}}_{n \ \text{factors}}\,
  &=\,h\V_n, \label{eqn:current}\\
  u|_{\bdy\Omega}\,&=\,\phi, \notag
\end{align}
where $\V_n$ is defined as
$$
  \V_n := (i/2)^n(dz_1\wedge d\zbar_1)\wedge\dots
  \wedge (dz_n\wedge d\zbar_n).
$$
When $u|_{\Omega}\notin \smoo^2(\Omega; \R)$, the left-hand side of
\eqref{eqn:current} must be interpreted as a current of bidegree $(n, n)$. That
this makes sense when $u\in {\rm psh}(\Omega)\cap \smoo(\overline\Omega)$
was established by Bedford--Taylor \cite{bedford-taylor:1976}.
\smallskip

Our objective in considering the above Dirichlet problem is as follows. With
$\Omega$, $h$, and $\phi$ as above, any solution of this problem is
a function $u\in {\rm psh}(\Omega)\cap \smoo(\overline\Omega)$ that satisfies
$u|_{\bdy\Omega} = \phi$; we would like to establish that there exist functions
with the latter properties that belong to some H{\"o}lder class on
$\overline\Omega$\,---\,assuming that $\bdy\Omega$ is sufficiently ``nice'' and
$\phi$ is sufficiently regular. The regularity theory for the complex
Monge--Amp{\`e}re equation provides us the means to the latter end. A
regularity theorem of the type hinted at for $\Omega$ strongly pseudoconvex
was established by Bedford--Taylor \cite[Theorem~9.1]{bedford-taylor:1976}.
Such theorems are much harder to deduce when $\Omega$ is \emph{weakly}
pseudoconvex. One such theorem is a special case of a result by Ha--Khanh
\cite{ha-khanh:2015}. Recall that $\langle\bcdot\,, \bcdot\rangle$
denotes the standard Hermitian inner product on $\Cn$.

\begin{result}[special case of {\cite[Theorem~1.5]{ha-khanh:2015}}]
\label{r:Holder_reg}
Let $\Omega\subset \Cn$, $n\geq 2$, be a bounded pseudoconvex domain having
$\smoo^2$-smooth boundary, let $\rho$ be a defining function of $\Omega$, and
let $m\geq 2$. Suppose 
\begin{itemize}
  \item[$(*)$] there
  exists a neighbourhood $U$ of $\bdy\Omega$, constants $c, C > 0$ and, for
  each $\delta > 0$ sufficiently small, there exists a plurisubharmonic
  function $\varphi_{\delta}$ on $U$ of class $\smoo^2$ such that 
  $|\varphi_{\delta}|\leq 1$ and such that
  \begin{align}
  \langle v, (\cHess\varphi_{\delta})(z)v\rangle &\geq 
  c\,(1/\delta)^{2/m}\|v\|^2 \quad \forall v\in \Cn, \label{eqn:Hess_lb}\\
  \|D\varphi_{\delta}(z)\| &\leq C/\delta, \label{eqn:grad_ub}
  \end{align}
  for each $z\in \rho^{-1}((-\delta, 0))$.
\end{itemize}
Let $\phi\in \smoo^{s,\,\alpha}(\bdy\Omega)$, $s=0, 1$, $\alpha\in (0, 1]$.
Then, the Dirichlet problem
\begin{align*}
  (dd^c{u})^n &= 0, \\
  u|_{\bdy\Omega} &= \phi, \notag
\end{align*}
has a unique plurisubharmonic solution $u\in
\smoo^{0,\,(s+\alpha)/m}(\overline\Omega)$.
\end{result}

The notation $\smoo^{j,\,\beta}$, $j\in \N$,
$\beta\in (0, 1]$, denotes the class of all \textbf{real}-valued functions
that are continuously differentiable to order $j$ (the latter being
suitably interpreted for the underlying space when $j\geq 1$) and
whose $j$-th partial derivatives satisfy a uniform H{\"o}lder condition with
exponent $\beta$. In what follows, if $j=0$ and $\beta\in (0,1)$, we shall
denote this class simply as $\smoo^{\beta}$.
\smallskip

The following result provides the connection between Result~\ref{r:Holder_reg}
and the condition on the Levi form in Theorems~\ref{th:kob_lower_bnd}
and~\ref{th:unbdd_visibile}.

\begin{result}[special case of {\cite[Theorem~2.1]{khanh-zampieri:2021}}]
\label{r:geometric}
Let $\Omega \subset \Cn$, $n\geq 2$, be a bounded pseudoconvex domain having
$\smoo^2$-smooth boundary. Assume there exists a $\smoo^2$-smooth closed
submanifold $S$ of $\bdy\Omega$ such that $S$ is totally-real and such that
$w(\bdy \Omega)\subset S$. Suppose there exists a number $m > 2$ such that
\begin{equation} \label{eqn:levi_degn}
  \Levi{\Omega}(\xi;v) \gtrsim \dst(\xi, S)^{m-2}\|v\|^2 
  \quad \forall v\in H_{\xi}(\bdy\Omega) 
  \text{ and } \forall \xi \in \bdy{\Omega}\setminus S.
\end{equation}
Then, $\Omega$ satisfies the condition $(*)$ in Result~\ref{r:Holder_reg}.
\end{result}

\begin{remark}
Some comments about Result~\ref{r:geometric} are in order. Firstly, 
\cite[Theorem~2.1]{khanh-zampieri:2021} is stated for $q$-pseudoconvex domains
satisfying a somewhat more general condition than \eqref{eqn:levi_degn}.
Result~\ref{r:geometric} is obtained by taking:
\begin{itemize}
  \item $q=1$, and
  \item $F(t) = ct^m$, $t > 0$, for some $c>0$,
\end{itemize}
in \cite[Theorem~2.1]{khanh-zampieri:2021}. (It must be noted that there is
a small typo in the description of $F$ in \cite{khanh-zampieri:2021}; the
asymptotic behaviour required of $F$ is $F(\delta)/\delta^2 \searrow 0$ as
$\delta \searrow 0$ and not what is stated on
\cite[p.~2769]{khanh-zampieri:2021}.) Secondly, the proof in
\cite{khanh-zampieri:2021} establishes just the estimate \eqref{eqn:Hess_lb}
(which is condition (1.5) in \cite{khanh-zampieri:2021}). However, it is evident
from the expression for $\varphi_{\delta}$ given that, since $m > 2$, the
estimate \eqref{eqn:grad_ub} is satisfied.
\end{remark}

We require one last result for proving Theorems~\ref{th:kob_lower_bnd}
and~\ref{th:unbdd_visibile}. 

\begin{result}[paraphrasing {\cite[Proposition~6]{sibony:1981}}]
\label{r:Kob_low_bd_Sibony}
Let $\Omega \subset \Cn$ be a domain and $z\in \Omega$. If there exists a negative
plurisubharmonic function $u$ on $\Omega$ that is of class $\mathcal{C}^2$
in a neighbourhood of $z$ and satisfies 
$$
  \langle v, (\cHess{u})(z) v \rangle \geq c \|v\|^2 \quad
  \forall v \in \Cn,
$$
for some $c > 0$, then 
$$
  k_{\Omega}(z;v) \geq \Big( \frac{c}{\alpha} \Big)^{1/2}\frac{\|v\|}{|u(z)|^{1/2}}
  \quad \forall v \in \Cn,
$$
where $\alpha > 0$ is a universal constant.
\end{result}

\section{Lower bounds for the Kobayashi metric} \label{sec:kob_lower}
We begin by stating and proving the general result relying on the complex
Monge--Amp{\'e}re equation to estimate the Kobayashi metric that was hinted at
in Section~\ref{ssec:Kobayashi}. Before we state it, we need a definition.

\begin{definition}
A function $\omega: [0,\infty)\to [0,\infty)$ is called a \emph{modulus of
continuity} if it is concave, monotone increasing, and such that
$\lim_{x\to 0^+}\omega(x) = \omega(0) = 0$.
\end{definition}

\begin{theorem} \label{th:kob_lower_bnd_general}
Let $\Omega \subset \Cn$, $n\geq 2$, be a bounded domain. Suppose there exists
a modulus of continuity $\omega: ([0, \infty), 0)\to ([0, \infty), 0)$ and
that, for each Lipschitz function $\phi: \bdy\Omega\to \R$, there
exists a function $u_{\phi}: \overline\Omega\to \R$ such that
$u_{\phi}|_{\Omega}$ solves the complex Monge--Amp{\`e}re equation
\begin{align*}
  (dd^c{u})^n &= 0, \\
  u|_{\bdy\Omega} &= \phi, \notag
\end{align*}
and satisfies
\begin{equation} \label{eqn:mod_cont}
  |u_{\phi}(z_1) - u_{\phi}(z_2)| \leq C_{\phi}\,\omega(\|z_1 - z_2\|)
  \quad \forall z_1, z_2\in \overline\Omega,
\end{equation}
for some constant $C_{\phi} > 0$. Then there exists a constant $c > 0$ such
that
\begin{equation} 
  k_{\Omega}(z; v) \geq 
  c\,\frac{\|v\|}
            {\omega(\Tdist(z))^{1/2}}
            \quad \forall z\in \Omega \text{ and } v \in \Cn.
\end{equation}
\end{theorem}

\begin{remark}
The hypothesis of Theorem~\ref{th:kob_lower_bnd_general} is essentially a
statement about the geometry of $\Omega$. It is well understood that the
boundary-regularity of the solutions of the complex Monge--Amp{\`e}re
equation is influenced by $\bdy\Omega$. Moreover, the existence
of $u_{\phi}$ too is constraint on $\Omega$: for instance, it rules out
those $\Omega$ that contain analytic varieties of positive dimension in
$\bdy\Omega$.
\end{remark}

Before proving Theorem~\ref{th:kob_lower_bnd_general} we state the
following elementary lemma.

\begin{lemma} \label{l:concave}
Let $\omega: [0,\infty)\to [0,\infty)$ be a concave, monotone increasing
function such that $\omega(0)=0$. Then, for all $\lambda, x\geq 0$, 
$\omega(\lambda x)\leq
(\lambda +1)\omega(x)$.
\end{lemma}

\begin{proof}[The proof of Theorem~\ref{th:kob_lower_bnd_general}]
Define $\phi: \bdy\Omega \to (-\infty, 0]$ by $\phi(z) := -2\|z\|^2$. 
We note that all the assertions below hold true trivially when
$\Omega$ is an Euclidean ball with centre $0\in \C^n$. As this
function is Lipschitz, there exists a function
$u_{\phi}: \overline{\Omega}\to \R$ with the properties stated in the
hypothesis of Theorem~\ref{th:kob_lower_bnd_general}. Let us define
$$
  {\Phi}(z)\,:=\,u_{\phi}(z)+\|z\|^2 \quad \forall z\in \overline\Omega.
$$

For $\nu\in \N$, write $\Omega_{\nu} := \{z\in \Omega : \Tdist(z) > 1/2^{\nu}\}$.
Let $\nu_0\in \Z_+$ and be so large that $\Omega_{\nu}$ is connected for every
$\nu\geq \nu_0$. It follows from \cite[Satz~4.2]{richberg:1968} by Richberg that
there exists a plurisubharmonic function $\Psi$ on $\Omega$ of class
$\smoo^\infty$ such that for all $\nu \geq \nu_0$
\begin{equation} \label{eqn:richberg}
  0\leq \Psi(z)-{\Phi}(z)\leq \omega(2^{-\nu}) \quad
  \forall z\in \Omega\setminus \Omega_{\nu}.
\end{equation}
Clearly, $\Psi$ extends continuously to $\overline\Omega$ (we shall refer to
this extension as $\Psi$ as well) and
\begin{equation} \label{eqn:Psi_bdy}
  \Psi(z) = -\|z\|^2 \quad \forall z\in \bdy\Omega.
\end{equation}
Now, let us write $U(z) := \Psi(z)+\|z\|^2$ for each $z\in \overline\Omega$.
Since $\Psi$ is plurisubharmonic,
\begin{equation} \label{eqn:u_Levi_pos}
  \langle v, (\cHess{U})(z) v \rangle \geq \|v\|^2 \quad \forall z\in \Omega
  \text{ and } v\in \Cn.    
\end{equation}

Fix a $z$ such that $\Tdist(z)\leq 1/2^{\nu_0}$.  As $\bdy\Omega$ is compact,
there exists
a point $\xi_z\in \bdy\Omega$ such that
$\Tdist(z) = \|z - \xi_z\|.$
There exists an integer $\nu_z\geq \nu_0$ such that
$$
  1/2^{(\nu_z+1)} < \Tdist(z)\leq 1/2^{\nu_z}.
$$
It follows from \eqref{eqn:richberg} that
\begin{align}
  |U(z)|\,&\leq\,|\Psi(z) - {\Phi}(z)| + |{\Phi}(z) + \|z\|^2| \notag \\
  &\leq\,\omega(2^{-\nu_z}) +
  |({\Phi}(z) + \|z\|^2)- ({\Phi}(\xi_z) + \|\xi_z\|^2)|. 
  \label{eqn:midway}
\end{align}
Now, owing to our hypothesis on $u_{\phi}$, there exists a constant
$C_1 >0 $ such that
$$
  |({\Phi}(z) + \|z\|^2)- ({\Phi}(\xi_z) + \|\xi_z\|^2)|
  \leq C_{\phi}\,\omega(\Tdist(z)) + C_1\Tdist(z).
$$
Here, we have used the condition \eqref{eqn:mod_cont} and
the fact that
$\|z-\xi_z\| = \Tdist(z)$. Combining the last estimate with
\eqref{eqn:midway}, we get, in view of Lemma~\ref{l:concave}:
\begin{align*}
  |U(z)|\,&\leq\,\left(\frac{2^{-\nu_z}}{\Tdist(z)}+1\right)\omega(\Tdist(z))
  + C_{\phi}\,\omega(\Tdist(z)) + C_1\Tdist(z) \\
  &\leq (3 + C_{\phi})\omega(\Tdist(z)) + C_1\Tdist(z).
\end{align*}
From the latter estimate, the fact that $\omega$ is concave, and that
$z$\,---\,apart from the constraint $\Tdist(z) \leq 1/2^{-\nu_0}\leq
1/2$\,---\,was chosen arbitrarily, we have
$$
 |U(z)| \leq C\omega(\Tdist(z)) \quad
 \forall z\in \Omega \text{ such that $\Tdist(z) \leq 1/2^{\nu_0}$}
$$
for some constant $C>0$. Since the set
$\{z\in \Omega:  \Tdist(z)\geq 1/2^{\nu_0}\}$ is compact, raising the value of
$C>0$ if needed, we get
\begin{equation} \label{eqn:u_bound}
  |U(z)| \leq C\omega(\Tdist(z)) \quad \forall z\in \Omega.  
\end{equation}

By \eqref{eqn:Psi_bdy}, we get $U|_{\bdy\Omega} = 0$. Thus, by the Maximum
Principle, $U$ is a smooth negative plurisubharmonic function. Thus, from
\eqref{eqn:u_Levi_pos}, \eqref{eqn:u_bound}, and
Result~\ref{r:Kob_low_bd_Sibony}, we conclude that
$$
  k_{\Omega}(z;v) \geq \Big( \frac{1}{C\alpha} \Big)^{1/2}\!\frac{\|v\|}
  {\omega(\Tdist(z))^{1/2}} \quad \forall z\in \Omega \text{ and } v \in \Cn,
$$
which is the desired lower bound. 
\end{proof}

A substantial part of the proof of Theorem~\ref{th:kob_lower_bnd} is the same as
that of the previous theorem. However, Theorem~\ref{th:kob_lower_bnd} is \emph{not}
a special case of Theorem~\ref{th:kob_lower_bnd_general}; the assumption on
$\bdy\Omega$ gives us better boundary behaviour of the solutions of the
same Dirichlet problem considered in the proof above. With these words, we give:

\begin{proof}[The proof of Theorem~\ref{th:kob_lower_bnd}]
Given our assumptions on $\bdy\Omega$, Result~\ref{r:geometric} implies that
$\Omega$ satisfies the condition $(*)$ in Result~\ref{r:Holder_reg}. As in
the proof of Theorem~\ref{th:kob_lower_bnd_general}, define
$\phi: \bdy\Omega \to (-\infty, 0]$ by $\phi(z) := -2\|z\|^2$. As
$\phi \in \smoo^{1,\,1}(\bdy\Omega)$, taking the values $s=1$ and
$\alpha = 1$ in the conclusion of Result~\ref{r:Holder_reg}, we see that
the Dirichlet problem stated in Result~\ref{r:Holder_reg}, with $\phi$ as
above, has a unique solution of class $\smoo^{2/m}(\overline\Omega)$. Let us
denote this solution by $u_{\phi}$. At this stage, exactly the same argument
as in the proof of Theorem~\ref{th:kob_lower_bnd_general} with
$$
  \omega(r) := r^{2/m}, \quad r\in [0, \infty),
$$
gives us a function $U$ defined on $\overline\Omega$ such that $U|_{\Omega}$
is a smooth negative plurisubharmonic function that satisfies the conditions
\begin{align}
  |U(z)| &\leq C\Tdist(z)^{2/m}, \label{eqn:size} \\
  \langle v, (\cHess{U})(z) v \rangle &\geq \|v\|^2 \label{eqn:Hessn} 
\end{align}
(for some constant $C>0$) for every $z\in \Omega$ and $v\in \Cn$. From these
inequalities and Result~\ref{r:Kob_low_bd_Sibony}, we conclude that
$$
  k_{\Omega}(z;v) \geq \Big( \frac{1}{C\alpha} \Big)^{1/2}\!\frac{\|v\|}
  {\Tdist(z)^{1/m}} \quad \forall z\in \Omega \text{ and } v \in \Cn,
$$
which is the desired lower bound.
\end{proof}

\begin{remark}\label{rem:MA-eqn_why}
The careful reader may notice that Theorem~\ref{th:kob_lower_bnd} could be
proved without
reference to Result~\ref{r:Holder_reg} or to the complex Monge--Amp{\`e}re
equation. Given Theorem~\ref{r:geometric}, one could instead appeal to
\cite[Theorem~2.1]{ha-khanh:2015} which provides a function that, suitably
modified, could substitute $U$ in the proof above. However, the
\textbf{proof} of \cite[Theorem~2.1]{ha-khanh:2015} involves a difficult
construction of a family of plurisubharmonic peak functions that must satisfy
very restrictive conditions. These are the ``plurisubharmonic peak functions
satisfying certain precise estimates'' alluded to in Section~\ref{ssec:Kobayashi}.
Indeed, a proof of Result~\ref{r:Holder_reg} can be given without the use of
such precise estimates. Furthermore, such (families of) plurisubharmonic peak
functions may not be available to large classes of domains, whereas there are
\textbf{many} approaches to the existence and boundary-regularity of solutions to the
homogeneous complex Monge--Amp{\`e}re equation.
Thus, the complex Monge--Amp{\`e}re equation may be a useful tool for establishing
estimates similar to \eqref{eqn:k-metric_lb}. The approach taken in the last proof,
and Theorem~\ref{th:kob_lower_bnd_general}, highlight the latter point.    
\end{remark}

\section{The proof of Theorem~\ref{th:unbdd_visibile}}
Before we can give the proof of Theorem~\ref{th:unbdd_visibile}, we give a
definition that will be useful in the latter proof.

\begin{definition}[Bharali--Zimmer, \cite{bharali-zimmer:2023}] \label{defn:local_M}
Let $\Omega \subset \Cn$ be a Kobayashi hyperbolic domain. Given a subset
$A \subset \overline{\Omega}$, we define the function $r \mapsto
M_{\Omega,\,A}(r) $, $r>0$, as
$$
  M_{\Omega,\,A}(r) := \sup \bigg\{ \frac{1}{k_{\Omega}(z;v)}: z \in
  A \cap \Omega, \, \delta_{\Omega}(z) \leq r, \, \|v\|=1 \bigg\}.
$$
\end{definition}

The function $M_{\Omega,\,A}$ is involved in one of the two conditions that a
point $p\in \bdy\Omega$, for $\Omega$ as in the above definition, must satisfy to
be what is called a ``local Goldilocks point'' by Bharali--Zimmer in
\cite{bharali-zimmer:2023}; see \cite[Definition~1.3]{bharali-zimmer:2023}.
The connection between local Goldilocks points and the visibility property is
given by the following

\begin{result}[paraphrasing {\cite[Theorem~1.4]{bharali-zimmer:2023}}]
\label{r:loc_Goldilocks_visi}
Let $\Omega\subset \Cn$ be a Kobayashi hyperbolic domain. If the set of points
in $\bdy\Omega$ that are not local Goldilocks points is a totally disconnected
set, then $\bdy\Omega$ is visible.
\end{result}

Another useful definition:
\begin{definition}
Let $\Omega$ be a domain in $\C^n$ and let $p\in \bdy\Omega$. A function
$\psi: \Delta \to (-\infty, 0]$, where $\Delta$ is a $\overline{\Omega}$-open
neighbourhood of $p$, is called a \emph{local plurisubharmonic peak function of
$\Omega$ at $p$} if $\psi\in {\rm psh}(\Delta\cap \Omega)\cap \smoo(\Delta)$
and satisfies
\[
  \psi(p) = 0 \quad\text{and} \quad \psi(z)<0 \quad 
  \forall z\in \Delta\setminus \{p\}.
\]
\end{definition}

We are now in a position to give

\begin{proof}[The proof of Theorem~\ref{th:unbdd_visibile}]
The proof of Theorem~\ref{th:unbdd_visibile} will be carried out in two steps.

\medskip
\noindent{{\textbf{Step 1.}} \emph{For $p\in \bdy\Omega$, constructing a
bounded subdomain $D_p$ such that $\bdy\Omega\cap \bdy{D_p} \ni p$ and is large}}%
\smallskip

\noindent{\textbf{Fix} $p\in \bdy\Omega$. Consider a unitary change of coordinate
$Z = (Z_1,\dots, Z_n)$ centered at $p$ (i.e., $Z(p) = 0$) with respect to which
$T_p(\bdy\Omega) = \{(Z_1,\dots, Z_n) \in \C^n: \im(Z_n) = 0\}$, the outward unit
normal to $\bdy\Omega$ at $p$\,($=0$) is $(0,\dots, 0, -i)$, and such that there
exist a neighbourhood $U^2_p := \rB{2n-1}(0, r_2)\times (-r_2, r_2)$ and a
function $\varphi_p : (\rB{2n-1}(0, r_2),0)\to (\R, 0)$ such
that $Z(\bdy\Omega)\cap U^2_p$ is connected and
$$
 Z(\Omega)\cap U^2_p\subset \{(Z', Z_n)\in \rB{2n-1}(0, r_2)\times\R :
 \im(Z_n) > \varphi_p(Z', \re(Z_n))\}
$$
(here, $r_2$ depends on $p$ but, for simplicity of notation, we will omit suffixes
and understand that this dependence is implied). Shrinking $r_2$ if necessary, we
will assume, additionally, that:
\begin{itemize}
  \item $\varphi_p(Z', x_n) \in (-r_2/3, r_2/3)$ for every 
  $(Z', x_n)\in \rB{2n-1}(0, r_2)$ and
  $$
    \{Z\in \rB{2n-1}(0, r_2)\times\R : \varphi_p(Z', \re(Z_n))
  < \im(Z_n) < \varphi_p(Z', \re(Z_n)) + r_2/3\} \subset Z(\Omega),
  $$
  
  \item $(Z(\bdy\Omega)\cap U^2_p) \cap Z(w(\bdy\Omega)) = \emptyset$ if
  $p\notin w(\bdy\Omega)$, and $U^2_p \Subset U_p$ if $p\in w(\bdy\Omega)$
\end{itemize}
(where $U_p$ is as in the statement of Theorem~\ref{th:unbdd_visibile}).
Fix $r_1\in (0, r_2/2)$. Let $\psi_p: \R^{2n-1}\to [0, \infty)$ be a
smooth, non-negative, radial convex function such that
\begin{align}
  \psi_p&|_{\rB{2n-1}(0, r_1)}\equiv 0, \text{ and} \notag \\
  \psi_p&|_{\R^{2n-1}\setminus \rB{2n-1}(0, r_1)}
  \text{ is strongly convex.} \label{eqn:str_cvx}
\end{align}
Clearly,
\begin{equation} \label{eqn:Levi-pscvx_graph}
  {\sf Graph}\!\left(
  \left.(\varphi_p + \psi_p)\right|
  _{\mathbb{B}^{2n-1}(0, r_2)\setminus \overline{\rB{2n-1}(0, r_1)}}\right)
  \text{ is a strongly Levi pseudoconvex hypersurface.}
\end{equation}}

Also, we can find a $\psi_p$ that satisfies all the above conditions and such that
$0\leq \psi_p < r_2/3$ on $\rB{2n-1}(0, r_2)$, due to which
\[
 \mathcal{S}_p := {\sf Graph}\!\left(
 \left.(\varphi_p + \psi_p)\right|_{\rB{2n-1}(0, r_2)}\right)
 \Subset Z(\Omega)\cup {\sf Graph}\left(\varphi_p|_{\overline{\rB{2n-1}(0, r_1)}}\right).
\]
Owing to this and to \eqref{eqn:Levi-pscvx_graph}, we can construct a bounded domain
$\wt{D}_p$ such that
\begin{enumerate}[leftmargin=25pt, label=$(\alph*)$]
  \item $\wt{D}_p\varsubsetneq Z(\Omega)$,
  \item $\mathcal{S}_p\subset \bdy{\wt{D}_p}$,
  \item $S_p := {\sf Graph}\left(\varphi_p|_{\overline{\rB{2n-1}(0, r_1)}}\right)
  = Z(\bdy\Omega)\cap
  \bdy{\wt{D}_p}$,
  \item $\bdy{\wt{D}_p}$ is strongly Levi pseudoconvex at each $\xi \in
  \bdy{\wt{D}_p}\setminus S_p$ whenever $p\in w(\bdy\Omega)$, and $\wt{D}_p$ is a
  strongly Levi pseudoconvex domain when $p\notin w(\bdy\Omega)$.
\end{enumerate}
Write $D_p := Z^{-1}(\wt{D}_p)$. Finally, we can extend
$S\cap Z^{-1}(S_p)$, whenever the latter is non-empty, to a $\smoo^2$-smooth
closed $1$-submanifold of $\bdy{D_p}$.
\medskip

\noindent{{\textbf{Step 2.}} \emph{Showing that each $p\in \bdy\Omega$ is a
local Goldilocks point.}}
\smallskip

\noindent{It is well-known that $p$ is a local Goldilocks point if
$p \notin w(\bdy \Omega)$. For a $p\in w(\bdy\Omega)$, it follows
from the discussion in the second paragraph of Step~1 and from the
properties $(a)$--$(d)$ that $D_p$ satisfies all the conditions of
Theorem~\ref{th:kob_lower_bnd} with $m=m_p$. Let
$\mathcal{U}_p: \overline{D}_p\to (-\infty, 0]$ denote the function
constructed in the proof of Theorem~\ref{th:kob_lower_bnd} that is
plurisubharmonic on $D_p$ and satisfies the conditions \eqref{eqn:size}, with
$m=m_p$, and \eqref{eqn:Hessn}. Let $W_p$ be a neighbourhood of $p$ having the
following properties (recall that $\bdy\Omega$ is $\smoo^2$-smooth):
\begin{itemize}
  \item $(\overline{W}_p\cap \Omega)\varsubsetneq D_p$ and
  $\overline{W_p\cap D_p}\cap \bdy{D_p}\varsubsetneq 
  \bdy{\Omega}\cap \bdy{D}_p$.
  \item $\Tdist(z) = \delta_{D_p}(z)$ for all $z \in W_p\cap D_p$.
  \item For each $z\in W_p$, there is a unique point 
  $\pi(z)\in \bdy\Omega\cap W_p$ such that $\delta_{D_p}(z)=\|z-\pi(z)\|$.
\end{itemize}
Fix $\xi\in \bdy\Omega\cap W_p$ and $\eps(\xi)\in (0, 1/2)$. Then, owing
to \eqref{eqn:Hessn}, the function
$$
  \psi_{\xi}(z) := \mathcal{U}_p(z) - \eps(\xi)\|z-\xi\|^{2},
  \quad z\in D_p\cup (\bdy\Omega\cap \overline{D}_p),
$$
is a local plurisubharmonic peak function of $\Omega$ at $\xi$. Applying a 
construction by Gaussier \cite[Section~2]{gaussier:1999}, we can find a
neighbourhood $W^{*}_p$ of $p$, $W^{*}_p\Subset W_p$, such that for each
$\xi\in \bdy\Omega\cap W^{*}_p$, there exist a neighbourhood $W_{\xi}$ of $\xi$,
$W_{\xi}\Subset W_p$, such that $\pi(z)\in \bdy\Omega\cap W_{\xi}$ for each
$z\in W_{\xi}\cap D_p$, and an upper-semicontinuous plurisubharmonic function
$\wt{\mathcal{U}}_{\xi}: \Omega\to (-\infty, 0)$ such that
\begin{equation} \label{eqn:agree}
  \wt{\mathcal{U}}_{\xi}(z) = \psi_{\xi}(z) \quad \forall 
  z\in W_{\xi}\cap D_p\,.
\end{equation}
Now, define $\wt{W}_p := \bigcup_{\xi\in\, (\bdy\Omega\cap W^{*}_p)\,}(W_{\xi}\cap W^{*}_p)$.
\smallskip

For $z_0\in \wt{W}_p\cap D_p$, by \eqref{eqn:agree},
Result~\ref{r:Kob_low_bd_Sibony}, and the inequality \eqref{eqn:Hessn} applied to
$\psi_{\pi(z_0)}$, we have
\[
  k_{\Omega}(z_0; v) \geq 
  \Big( \frac{1/2}{\alpha} \Big)^{1/2}\!\frac{\|v\|}
  {\big(\,|\mathcal{U}_{p}(z_0)|
    + \eps(\pi(z_0))\delta_{D_p}(z_0)^2\big)^{1/2}}
    \quad \forall v \in \Cn.
\]
The above inequality and \eqref{eqn:size} (taking $m = m_p$) imply that there exists a constant $C_p > 0$,
independent of $z_0\in \wt{W}_p\cap D_p$, such that
\begin{align*}
  k_{\Omega}(z_0; v) &\geq 
  \Big( \frac{1/2}{C_p\alpha} \Big)^{1/2}\!\frac{\|v\|}
  {\delta_{D_p}(z_0)^{1/m_p}} 
  = \Big( \frac{1/2}{C_p\alpha} \Big)^{1/2}\!\frac{\|v\|}{\Tdist(z_0)^{1/m_p}}  
            \quad \forall v \in \Cn.
\end{align*}
The equality above is because $\Tdist(z_0) = \delta_{D_p}(z_0)$.
Write $A^p := \wt{W}_p\cap \overline{\Omega}$. Since $z_0\in \wt{W}_p\cap D_p$ was
arbitrarily chosen, it follows
from the above estimate that the quantity $M_{\Omega,\,A^p}$ satisfies
\begin{equation} \label{eqn:lgp_condn-1}
  M_{\Omega,\,A^p}(r) \leq c_pr^{1/m_p} \quad\text{for all $r>0$ sufficiently
  small,}
\end{equation}
for some $c_p > 0$.}
\smallskip

By Lemma~\ref{l:loc-int-cone}, $\Omega$ satisfies a local interior cone
condition. Thus, by \cite[Lemma~2.2]{bharali-zimmer:2023} and
\eqref{eqn:lgp_condn-1} it follows that $p$ satisfies the conditions for being a
local Goldilocks point. Since $p\in w(\bdy\Omega)$ was chosen arbitrarily, it
follows that every boundary point is a local Goldilocks point.
Thus, by Result~\ref{r:loc_Goldilocks_visi}, $\bdy\Omega$ is visible.
\end{proof}
\smallskip

\section{The proof of Theorem~\ref{th:unbdd_Picard-type}} \label{sec:proof_Picard-type}
In this section, we give the (short) proof of
Theorem~\ref{th:unbdd_Picard-type}. This is a Picard-type extension theorem, as
discussed in Section~\ref{sec:intro}. Such a theorem relies upon
$\Omega\subset \Cn$ being hyperbolically imbedded, just as in Result~\ref{r:Picard-type},
even though $\Omega$ is not relatively compact. We state the result by
Joseph--Kwack alluded to in Section~\ref{sec:intro}, which clarifies
the latter statement. Here, $Z^\infty$ denotes the one-point compactification
of $Z$.

\begin{result}[Joseph--Kwack, {\cite[Corollary~7]{joseph-kwack:1994}}:
paraphrased for $Y$, $Z$ manifolds] \label{r:Picard-type_hyp-imb}
Let $Z$ be a complex manifold and $Y$ be a complex submanifold of $Z$
such that $Y$ is hyperbolically imbedded in Z. Let $X$ be a complex
manifold, let $k = \dim_{\C}(X)$, and let $\mathcal{A} \varsubsetneq X$ be an
analytic subvariety of dimension $(k-1)$ having at most normal-crossing
singularities. Then, any holomorphic map $f: X \setminus \mathcal{A} \rightarrow Y$
extends as a continuous map $\widetilde{f}:X \rightarrow Z^{\infty}$.
\end{result}

\begin{remark}
A comment on some terminology and definitions used in Joseph--Kwack \cite{joseph-kwack:1994} are in order.
With $Y, Z$ as above, they defined a notion of when a point in $\overline{Y} \subset Z$ is a \emph{hyperbolic point
for Y} (see \cite[p.~363]{joseph-kwack:1994} for the definition). Many of the foundational results in
\cite{joseph-kwack:1994} give certain necessary and sufficient conditions for a point in $\overline{Y}$
to be a hyperbolic point for $Y$. A careful reading of the proofs in \cite{joseph-kwack:1994} indicates that
the proof of \cite[Corollary~7]{joseph-kwack:1994} relies on the latter results, and its conclusion holds true
under the assumption that each point in $\overline{Y}$ is a hyperbolic point for $Y$. It turns out
that $Y$ is hyperbolically imbedded in $Z$ (in the sense of
Definition~\ref{defn:hyp_imb}) if and only if each
point in $\overline{Y}$ is a hyperbolic point for $Y$; see \cite{joseph-kwack:preprint},
\cite[Corollary~2]{thai-duc-minh:2005}. 
\end{remark}

\begin{proof}[Proof of Theorem~\ref{th:unbdd_Picard-type}]
Since $\Omega$ satisfies the hypothesis of Theorem~\ref{th:unbdd_visibile},
$\bdy \Omega$ is visible. So, owing to Proposition~\ref{prpn:visi_hyp-im},
$\Omega$ is hyperbolically imbedded in $\Cn$. Thus, by
Result~\ref{r:Picard-type_hyp-imb}, the proof follows immediately.
\end{proof}

\subsection{An example} We conclude this discussion with a basic example
which shows that, for $X$, $\mathcal{A}$ as in Section~\ref{sec:intro}, a
holomorphic map $f: X \setminus\mathcal{A}\rightarrow \Omega$, where $\Omega$
is a hyperbolically imbedded domain, does not, in general, extend continuously
to $X$ if the singularities of $\mathcal{A}$ are even slightly worse than
normal-crossing singularities.

\begin{example}\label{ex:not_ext}
An example of an unbounded planar domain $\Omega$ that is hyperbolically
imbedded in $\C$ and a holomorphic function
$f: (\D^2 \setminus \mathcal{A}) \rightarrow \Omega$, where $\mathcal{A}$
is a closed analytic set in $\D^2$ of codimension $1$ containing singular points,
but not normal-crossing singularities, such that $f$ does \textbf{not} extend:
\begin{itemize}
  \item either to a holomorphic function on $\D^2$,
  \item or to a continuous map from $\D^2$ to $\C^{\infty}$.
\end{itemize}
\end{example}

\noindent{Let $\Omega := \C \setminus \{0, 1\}$. It is a long-established
fact that $\cproj{1} \setminus \{[0:1], [1:0], [1:1]$\} is hyperbolically
imbedded in $\cproj{1}$. So, it follows by
Lemma~\ref{l:hyp_imb_intermediate} that $\Omega$ is hyperbolically
imbedded in $\C$. Define
$$
  \mathcal{A} := \{(z,w) \in \D^2 : z(w-z)w=0\}.
$$
If we define $f: (\D^2 \setminus \mathcal{A}) \rightarrow \C$ by
$$
  f(z,w) := z/w \quad \forall (z,w)\in (\D^2 \setminus \mathcal{A}),
$$
then it is elementary to see that, by construction, $f$ is holomorphic and
that ${\sf range}(f) \subseteq \Omega$. For each fixed $\lambda \in
\C \setminus \{0, 1\}$, $(\lambda\zt, \zt) \in
\D^2 \setminus \mathcal{A}$ for every $\zt\in \C^*$ with sufficiently
small $|\zt|$. We have
$$
  \lim_{\zt \to 0}f(\lambda\zt, \zt) = \lambda.
$$
Since $\lambda \in \C \setminus \{0, 1\}$ was arbitrary, the above shows that
$(0,0)$ is a point of indeterminacy of $f$. Hence, the extension of $f$
to $\D^2$ in either of the two above-mentioned ways is impossible.
\hfill $\btl$}
\medskip

\section*{Acknowledgements}
A.~\!Banik is supported in part by a scholarship from the National Board for Higher
Mathematics (NBHM) (Ref. No.~0203/16(19)/2018-R\&D-II/10706) and by financial
assistance from the Indian Institute of Science.  Both A.~\!Banik and G.~\!Bharali
are supported by the DST-FIST programme (grant no. DST FIST-2021 [TPN-700661]).

\end{document}